\documentclass[12pt]{amsart}

\usepackage{amsmath,amssymb,amsthm,mathtools}

\usepackage{xcolor}
\newtheorem{theorem}{Theorem}[section]
\newtheorem{proposition}[theorem]{Proposition}
\newtheorem{lemma}[theorem]{Lemma}
\newtheorem{corollary}[theorem]{Corollary}
\newtheorem{definition}[theorem]{Definition}
\newtheorem{example}[theorem]{Example}
\theoremstyle{remark}

\newcommand{\sa}{\mathrm{sa}}

\title[Order isomorphisms on  positive cones of $C^*$-algebras]
{Order isomorphisms on positive cones of non-unital $C^*$-algebras}

\author{Osamu Hatori}
\address{Niigata University,
Niigata 950-2181, Japan}
\email{oppekepenguin@gmail.com}

\author{Shiho Oi}
\address{Department of Mathematics, Faculty of Science, Niigata University,
Niigata 950-2181, Japan}
\email{shiho-oi@math.sc.niigata-u.ac.jp}

\thanks{The second author was supported in part by JSPS KAKENHI Grant Number JP24K06754.
}
\dedicatory{Dedicated to Professor Ngai-Ching Wong on the occasion of his 65th birthday}

\begin{document}

\begin{abstract}
Let $A$ and $B$ be $C^*$-algebras, not necessarily unital, and let
$T:A_+\to B_+$ be a positively homogeneous order isomorphism.  We first
prove that $T$ is additive and extends uniquely to a bounded real-linear
order isomorphism between the self-adjoint parts of $A$ and $B$.  Passing
to the biduals, we then obtain a representation theorem for positively homogeneous order isomorphisms. Finally, we study surjective maps between the positive cones of $C^*$-algebras, again not necessarily unital, which preserve the norm of the arithmetic mean. This extends a previous result by removing the assumption that at least one of the algebras is unital.
\end{abstract}

\maketitle
\section{Introduction}

Let $A$ and $B$ be $C^*$-algebras, not necessarily unital.  We denote by $A_+$ and $B_+$ the sets of all positive elements of $A$ and $B$, respectively; these sets are called the positive cones of $A$ and $B$. The study of nonlinear maps on positive cones that determine the
Jordan structure of the underlying $C^*$-algebras has developed in
several directions. 
In the unital setting, Moln\'ar proved that every positively homogeneous order isomorphism between the sets of positive invertible elements of two $ C^*$-algebras is induced by a Jordan $*$-isomorphism and a positive invertible weight \cite[Proposition~13]{MolnarRenyi}. Recall that a bijection $T:A_+\to B_+$ is called a positively homogeneous order isomorphism if
\[
        a\le b\iff T(a)\le T(b),
        \qquad
        T(\lambda a)=\lambda T(a)
        \quad(a\in A_+,\ \lambda>0).
\]
A corresponding result for the entire positive cones of unital $C^*$-algebras was obtained by the authors.
\begin{theorem}\cite[Theorem~5.1]{HatoriOiPositiveCones}\label{unital}
    Let $A$ and $B$ be unital $C^{*}$-algebras. If $T: A_+\to B_+$ is a positively homogeneous order isomorphism, then there is a Jordan $*$-isomorphism $J:A \to B$ such that $T(a)=T(1)^{\frac{1}{2}}J(a)T(1)^{\frac{1}{2}}$ for any $a \in A_{+}$.  
\end{theorem}
In Section \ref{section2}, we extend Theorem~\ref{unital} to the setting of general $C^*$-algebras (not assumed to be unital).

As an application of \cite[Proposition~13]{MolnarRenyi}, Dong, Li, Moln\'ar and Wong
\cite{DongLiMolnarWong} studied surjective transformations between positive cones of unital $C^*$-algebras. In Theorem 2.5 of \cite{DongLiMolnarWong}, they characterize surjective maps between the positive cones of $C^*$-algebras that preserve the norm of the arithmetic mean, under the assumption that at least one of the algebras is unital. It is natural to ask whether the unitality assumption is essential. In this paper, we show that it can be removed completely. More precisely, we establish the same characterization for arbitrary $C^*$-algebras, with neither algebra assumed to be unital. Thus, Theorem \ref{additive} extends Theorem 2.5 of \cite{DongLiMolnarWong} to the non-unital setting. 

We first fix the multiplier terminology.
\begin{definition}\label{def:QM}
Let $C$ be a $C^*$-algebra, canonically embedded in $C^{**}$, which is the bidual of $C$.  Its
multiplier algebra is
\[
        M(C)
        =
        \{x\in C^{**}:xC\subset C,\ Cx\subset C\},
\]
and its quasi-multiplier space is
\[
        QM(C)
        =
        \{x\in C^{**}:axb\in C
          \text{ for all }a,b\in C\}.
\]
Thus $M(C)\subset QM(C)$. 
The inclusion may be proper. It is well-known that $M(C)$ is a unital $C^{*}$-subalgebra of $C^{**}$. We denote the unit of $M(C)$ by $1_{C^{**}}$.  Unlike $M(C)$, the space $QM(C)$ need
not be closed under multiplication. We regard $C$, $M(C)$, and $QM(C)$ as subspaces of $C^{**}$. 
\end{definition}

\section{Positively homogeneous order isomorphisms}\label{section2}
We study positively homogeneous order isomorphisms on the positive cones of $C^{*}$-algebras.  The key input is Sch\"affer's order-unit linearization theorem \cite[Theorem~B]{SchafferGauge}, in the
form recalled in \cite[Theorem~5.2]{LGI}.
Fix $u\in A_+\setminus\{0\}$ and put
\[
 X_u=\{x\in A_{\sa}: -\lambda u\le x\le \lambda u
       \text{ for some }\lambda>0\},
 \qquad
 C_u=X_u\cap A_+.
\]
Then $X_u$ is a real partially ordered vector space with positive cone
$C_u$.  By the definition of $X_u$, the element $u$ is an order unit.
Moreover, $(X_u,C_u,u)$ is Archimedean.  Indeed, suppose that $x\in X_u$
and $y\in C_u$ satisfy $nx\le y$ for every $n\ge1$.  Since
$y\le\|y\|1_{A^{**}}$, we have
\[
        nx\le\|y\|1_{A^{**}} \qquad(n\ge1).
\]
If $x\not\le0$, then $\sup\sigma(x)>0$, and hence
$n\sup\sigma(x)\le\|y\|$ for every $n$, a contradiction.  Thus $x\le0$.
We write
\[
 \|x\|_u=\inf\{\lambda>0:-\lambda u\le x\le\lambda u\}
\]
for the corresponding order-unit norm.

\begin{lemma}\label{lem:interior}
The interior of $C_u$ with respect to the order-unit norm is
\[
 C_u^\circ
 =\{x\in A_+: \alpha u\le x\le\beta u
       \text{ for some }\alpha,\beta>0\}.
\]
\end{lemma}

\begin{proof}
Suppose first that $\alpha u\le x\le\beta u$ for some
$\alpha,\beta>0$.  If $\|y-x\|_u<\alpha/2$, choose
$\delta<\alpha/2$ such that
\[
        -\delta u\le y-x\le\delta u.
\]
Then $y\ge(\alpha-\delta)u\ge(\alpha/2)u\ge0$.  Thus
$x\in C_u^\circ$.

Conversely, if $x\in C_u^\circ$, then there is $\varepsilon>0$ such
that $x+y\in C_u$ whenever $\|y\|_u<\varepsilon$.  Taking
$y=-(\varepsilon/2)u$ gives
\[
 \frac\varepsilon2u\le x.
\]
Since $x\in X_u$, there is $\beta>0$ with $x\le\beta u$.
\end{proof}

\begin{lemma}\label{lem:span-interior}
We have $X_u=\operatorname{span}_{\mathbb R} C_u^\circ$. 
\end{lemma}

\begin{proof}
Let $x\in X_u$.  Choose $\lambda>0$ with
$-\lambda u\le x\le\lambda u$.  Then
\[
 u\le x+(\lambda+1)u\le(2\lambda+1)u,
\]
so $x+(\lambda+1)u\in C_u^\circ$ by
Lemma~\ref{lem:interior}.  Also $(\lambda+1)u\in C_u^\circ$, whence
\[
 x=\bigl(x+(\lambda+1)u\bigr)-(\lambda+1)u
 \in\operatorname{span}_{\mathbb R}C_u^\circ.
\]
The reverse inclusion is immediate.
\end{proof}

\begin{theorem}\label{thm:extension}
Let $A$ and $B$ be $C^*$-algebras and let $ T: A_+\to B_+$ be a positively homogeneous order isomorphism. 
Then $T$ is additive on $A_+$ and extends uniquely to a real-linear
order isomorphism
\[
 \widetilde T:A_{\sa}\longrightarrow B_{\sa}.
\]
\end{theorem}

\begin{proof}
Since $0$ is the least element of each positive cone and $T$ is an
order isomorphism, $T(0)=0$.  In particular, $u\ne0$ implies
$T(u)\ne0$.

Fix $u\in A_+\setminus\{0\}$.  By Lemma~\ref{lem:interior}, since $T$ and $T^{-1}$ are 
order-preserving and positively homogeneous, 
\[
 T(C_u^\circ)=C_{T(u)}^\circ.
\]
Hence
\[
 \Phi_u:=T|_{C_u^\circ}:C_u^\circ\longrightarrow C_{T(u)}^\circ
\]
is a positively homogeneous order isomorphism.  By Sch\"affer's
order-unit linearization theorem, in the form recalled in
\cite[Theorem~5.2]{LGI}, $\Phi_u$ is linear in the sense that it is the
restriction of a real-linear map on the span of $C_u^\circ$.  Since
$C_u^\circ$ and $C_{T(u)}^\circ$ span $X_u$ and $X_{T(u)}$, respectively,
by Lemma~\ref{lem:span-interior}, $\Phi_u$ extends uniquely to a
real-linear map
\[
 \Phi_u:X_u\longrightarrow X_{T(u)}.
\]
We claim that $\Phi_u(x)=T(x)$ for every $x\in C_u$.  
Fix $x\in C_u$.  We first prove that $\Phi_u(x)\in B_+$. As $x\in C_u$, $x+\varepsilon u\in C_u^\circ$ for every $\varepsilon>0$, and hence
\[
        \Phi_u(x)+\varepsilon\Phi_u(u)
        =\Phi_u(x+\varepsilon u)\in C_{T(u)}\subset B_+.
\]
Letting $\varepsilon\downarrow0$ and using the norm-closedness of $B_{+}$, we obtain $\Phi_u(x)\in B_+$. 
Next we prove $T(x)\le \Phi_u(x)$. 
Choose $\lambda>0$ with $0\le x\le\lambda u$.  For every
$\varepsilon>0$,
\[
 \varepsilon u\le x+\varepsilon u
 \le(\lambda+\varepsilon)u,
\]
so $x+\varepsilon u\in C_u^\circ$.  Therefore
\[
 T(x)\le T(x+\varepsilon u)
 =\Phi_u(x)+\varepsilon\Phi_u(u).
\]
Letting $\varepsilon\downarrow0$ and using the norm-closedness of $B_{+}$, we obtain 
$T(x)\le\Phi_u(x)$.

Since $\Phi_u(x)\in B_+$, we may apply $T^{-1}$.
From
\[
    0\le\Phi_u(x)
    \leq\Phi_u(x)+\varepsilon\Phi_u(u)
    =T(x+\varepsilon u)
\]
and the order preservation of $T^{-1}$, we obtain
\[
    T^{-1}(\Phi_u(x))\leq x+\varepsilon u
    \qquad(\varepsilon>0).
\]
Again letting $\varepsilon\downarrow0$ gives
$T^{-1}(\Phi_u(x))\le x$, and hence $\Phi_u(x)\le T(x)$.  Thus $T(x)=\Phi_u(x)$ for all $x\in C_u$.

Now let $a,b\in A_+$, not both zero, and put $u=a+b$.  Since
$a,b,a+b\in C_u$,
\[
 T(a+b)=\Phi_u(a+b)=\Phi_u(a)+\Phi_u(b)=T(a)+T(b).
\]
The case $a=b=0$ is trivial, so $T$ is additive on $A_+$.

For $x\in A_{\sa}$ write $x=a_1-a_2$ with $a_1,a_2\in A_+$ and set
\[
 \widetilde T(x)=T(a_1)-T(a_2).
\]
If $a_1-a_2=b_1-b_2$, then $a_1+b_2=b_1+a_2$, and additivity gives
\[
 T(a_1)+T(b_2)=T(b_1)+T(a_2),
\]
so the definition is independent of the chosen decomposition.
Additivity and positive homogeneity of $T$ imply real linearity of
$\widetilde T$.  The same construction applied to $T^{-1}$ shows that
$\widetilde T$ is bijective.

Finally, if $x\le y$, then $y-x\in A_+$ and therefore
\[
 \widetilde T(y)-\widetilde T(x)
 =\widetilde T(y-x)=T(y-x)\ge0.
\]
The same argument for $\widetilde T^{-1}$ shows that the order is
preserved in both directions.  Uniqueness follows because
$A_{\sa}=A_+-A_+$, so a real-linear extension is determined by its
values on $A_+$.
\end{proof}

We first record the boundedness of the extension and the order
properties of its second adjoint.

\begin{lemma}\label{lem:bounded-bidual}
The real-linear order isomorphism
$\widetilde T:A_{\sa}\to B_{\sa}$ obtained in
Theorem~\ref{thm:extension} is bounded and has bounded inverse.
Its complex-linear extension, denoted by the same symbol,
\[
    \widetilde T:A\longrightarrow B,\qquad
    \widetilde T(a+ib)=\widetilde T(a)+i\widetilde T(b)
    \quad(a,b\in A_{\sa}),
\]
is a bounded, $^*$-preserving complex-linear bijection with bounded
inverse. Moreover,
\[
    \widetilde T^{**}:A^{**}\longrightarrow B^{**}
\]
is a weak$^*$-continuous complex-linear order isomorphism, with
\[
    (\widetilde T^{**})^{-1}
    =(\widetilde T^{-1})^{**}.
\]
\end{lemma}

\begin{proof}
Write $L=\widetilde T|_{A_{\sa}}$.
We first show that $L$ is bounded on the positive unit ball.
Otherwise, by positive homogeneity, we could choose $a_n\in A_+$
such that
\[
    \|a_n\|\le 2^{-n},
    \qquad
    \|L(a_n)\|\ge n
    \quad(n\ge1).
\]
The series $a=\sum_{n=1}^{\infty}a_n$ converges in norm to an
element of $A_+$, and $a_n\le a$ for every $n$.
Since $L$ is positive, this gives
\[
    0\le L(a_n)\le L(a),
\]
and hence
\[
    n\le\|L(a_n)\|\le\|L(a)\|
    \quad(n\ge1),
\]
a contradiction.

Thus there is a constant $M>0$ such that
$\|L(a)\|\le M\|a\|$ for every $a\in A_+$.
For $x\in A_{\sa}$, its positive and negative parts belong to
$A_+$ and satisfy $\|x_+\|,\|x_-\|\le\|x\|$.
Consequently,
\[
    \|L(x)\|
    \le\|L(x_+)\|+\|L(x_-)\|
    \le 2M\|x\|.
\]
The same argument applies to the positive real-linear map $L^{-1}$.

The complex-linear extension of $L$ is $^*$-preserving.
For $x=a+ib$, where $a,b\in A_{\sa}$, we have
$\|a\|,\|b\|\le\|x\|$, so
\[
    \|\widetilde T(x)\|
    \le\|L(a)\|+\|L(b)\|
    \le 2\|L\|\|x\|.
\]
The complex-linear extension of $L^{-1}$ is bounded by the same
argument and is the inverse of $\widetilde T$.
Both $\widetilde T$ and $\widetilde T^{-1}$ are positive.

We next verify positivity of the second adjoint.
Let $x\in A^{**}_+$ and $\varphi\in B^*_+$.
Since $\widetilde T$ is positive, the functional
$\widetilde T^*(\varphi)=\varphi\circ\widetilde T$ belongs to $A^*_+$.
Therefore,
\[
    \langle\widetilde T^{**}(x),\varphi\rangle
    =\langle x,\widetilde T^*(\varphi)\rangle
    \ge0.
\]
The characterization of the positive cone of $B^{**}$ by positive
functionals in $B^*$ now gives
$\widetilde T^{**}(x)\in B^{**}_+$.
The same argument shows that $(\widetilde T^{-1})^{**}$ is positive.

Finally, taking second adjoints of the identities
$\widetilde T^{-1}\widetilde T=\operatorname{id}_A$ and
$\widetilde T\widetilde T^{-1}=\operatorname{id}_B$ gives
\[
    (\widetilde T^{-1})^{**}\widetilde T^{**}
    =\operatorname{id}_{A^{**}},
    \qquad
    \widetilde T^{**}(\widetilde T^{-1})^{**}
    =\operatorname{id}_{B^{**}}.
\]
Thus $\widetilde T^{**}$ is an order isomorphism.
Its weak$^*$ continuity follows from its being the adjoint of
$\widetilde T^*$.
\end{proof}

Applying Kadison's theorem \cite{Kadison}, we have the following.
\begin{proposition}\cite[Remark 4.3(ii)]{BrownOrder}\label{BrownO}
Let $T\colon A\to B$ be a complex-linear order
isomorphism, and set
 $s:=T^{**}(1_{A^{**}})$. 
Then $s$ is positive and invertible, and $s^{-1}\in QM(B)$. There exists a Jordan
${*}$-isomorphism $J$ from $A$ onto $s^{-\frac{1}{2}}Bs ^{-\frac{1}{2}}$ 
such that \[
T=R_{s} J,
\]
where $ R_s:s^{-\frac{1}{2}}Bs ^{-\frac{1}{2}} \to B$ is defined by  $R_s(x):=s^{1/2}xs^{1/2}$. 

\end{proposition}
Fix $s \in QM(B)$ be a positive invertible element in $B^{**}$. Following \cite[Proposition~3.1]{BMS}, let $B^{s}$ be $B$ as a
complex vector space, equipped with the original involution and the
new multiplication \[
    a\mathbin{\bullet_s} b := asb
    \qquad a,b\in B.
\]
We endow $(B^s, \mathbin{\bullet_s})$  with the equivalent norm
\[
    \|a\|_s := \|s\|\,\|a\|.
\]
Then $(B^s, \mathbin{\bullet_s})$ is a Banach $*$-algebra. We define the Jordan product of $a, b \in  (B^s, \mathbin{\bullet_s})$ by $a \circ b=\frac{1}{2}(a\mathbin{\bullet_s}b+b\mathbin{\bullet_s} a)$. 
\begin{corollary} Every complex-linear order isomorphism $T\colon A\to B$ is a Jordan $*$-isomorphism from $A$ onto $(B^{s^{-1}}, \mathbin{\bullet_{s^{-1}}})$. 
\end{corollary}
\begin{proof}
    By Proposition \ref{BrownO}, there is a Jordan $*$-isomorphism $J$ such that $T(a)=s^{\frac{1}{2}}J(a)s^{\frac{1}{2}}$ for any $a \in A$. We have
    \begin{equation*}
\begin{split}
    s^{-\frac{1}{2}}T(a \circ b)s^{-\frac{1}{2}}&=J(a \circ b)=J(a) \circ J(b)
    =\frac{1}{2}(J(a)J(b)+J(b)J(a))\\
    &=\frac{1}{2} s^{-\frac{1}{2}}(s^{\frac{1}{2}}J(a)s^{\frac{1}{2}}s^{-1}s^{\frac{1}{2}}J(b)s^{\frac{1}{2}}+s^{\frac{1}{2}}J(b)s^{\frac{1}{2}}s^{-1}s^{\frac{1}{2}}J(a)s^{\frac{1}{2}})s^{-\frac{1}{2}}\\
    &=\frac{1}{2} s^{-\frac{1}{2}}(T(a)\circ T(b)+T(b)\circ T(a))s^{-\frac{1}{2}}.
\end{split}
    \end{equation*}
    For any $a,b\in A$, we have 
    $T(a\circ b)
    =
    T(a)\circ T(b)$.
\end{proof}

\begin{lemma}\label{equi}
     We have 
     \[
     J(A)=B \Longleftrightarrow  \widetilde{T}^{**}(1_{A^{**}}) \in M(B)
     \]
\end{lemma}
\begin{proof}
   Let  $s:=\widetilde{T}^{**}(1_{A^{**}})$. By Proposition \ref{BrownO}, we have $s^ {-1}\in QM(B)$. \\
   ($\Rightarrow$) By Proposition \ref{BrownO}, we have $s^{-\frac{1}{2}}Bs ^{-\frac{1}{2}}=J(A)=B$. We write $h:=s^{\frac{1}{2}}>0$. Then we have $h^{-1}Bh^{-1}= B$. 
Let $(e_\lambda)_\lambda$ be a contractive approximate identity for
$B$. Fix $b\in B$ and put $x=bs^{-1}$. Since $s^{-1} \in QM(B)$, we have
$xe_\lambda=bs^{-1}e_{\lambda} \in B$ for every
$\lambda$. Furthermore, since $s^{-1}$ is positive, $ x^*x=s^{-1} b^*b s^{-1}=h^{-1}(h^{-1}b^{*}bh^{-1})h^{-1}\in B$. 
It follows that
\[
\begin{aligned}
    \lVert x-xe_\lambda\rVert^2
    &=\| (1-e_\lambda)x^*x(1-e_\lambda)\| \le 2\|(1-e_\lambda)x^{*}x\|
      \longrightarrow 0.
\end{aligned}
\]
Since $B$ is norm closed and $xe_\lambda \in B$, this shows that $x=bs^{-1} \in B$. Therefore,
$Bs^{-1}\subseteq B$. By a similar argument, we also have $s^{-1}B\subseteq B$. This implies $s^{-1}\in M(B)$, so $s \in M(B)$.  
\\
($\Leftarrow$) Applying Proposition \ref{BrownO}, we have $J(A)=s^{-\frac{1}{2}}Bs^{-\frac{1}{2}}$. Assume $s \in M(B)$. As $M(B)$ is a $C^{*}$-subalgebra of $B^{**}$, we have $s^{-\frac{1}{2}},  s^{\frac{1}{2}} \in M(B)$. Thus $s^{-\frac{1}{2}}Bs^{-\frac{1}{2}} \subset B$ and $s^{\frac{1}{2}}Bs^{\frac{1}{2}} \subset B$. This implies $J(A)=s^{-\frac{1}{2}}Bs^{-\frac{1}{2}}=B$.
\end{proof}

If $B$ is unital, then $\widetilde{T^{**}}(1_{A^{**}}) \in B=M(B)$. Thus Lemma \ref{equi} yields $J(A)=B$ holds. On the other hand, $J(A)=B$ may fail when the algebras are
non-unital.

\begin{example}\label{ex:failure}
Let $A$ be a $C^*$-algebra for which $QM(A)\ne M(A)$.
There is $s\in QM(A)\setminus M(A)$ such that $s$ is positive and invertible in $A^{**}$.  
Define $B=s^{1/2}As^{1/2}\subset A^{**}$. It is well-known that $B$ is a $C^*$-algebra; compare \cite[Corollary~3.3]{BMS}.

Define $T:A_+\to B_+$ by $T(a)=s^{1/2}as^{1/2}$. 
Then $T$ is a positively homogeneous order isomorphism.
The same formula defines a bounded complex-linear bijection $\widetilde T:A \to B$ by $\widetilde T(a)=s^{1/2}as^{1/2}$, whose restriction to $A_+$ is $T$. Since $s$ is invertible, its support projection is $1$.  By \cite[Proposition~3.2]{BMS}, the enveloping von Neumann algebra of
$\overline{s^{1/2}As^{1/2}}$ is therefore canonically $A^{**}$. 
Thus, under the canonical identification $B^{**}\cong A^{**}$, we have $\widetilde T^{**}(a)=s^{1/2}as^{1/2}$ for any $a \in A^{**}$. 
 Thus $\widetilde T^{**}(1_{A^{**}})=s$.
We claim that $s\notin M(B)$.
Suppose, to the contrary, that $s\in M(B)$.
Since $s$ is positive and invertible in $B^{**}$ and $M(B)$ is a
unital $C^*$-subalgebra of $B^{**}$, continuous functional calculus
gives $s^{1/2},s^{-1/2}\in M(B)$. 
Consequently,
\[
    s^{-1/2}Bs^{-1/2}\subset B
    \quad\text{and}\quad
    s^{1/2}Bs^{1/2}\subset B.
\]
Thus we get $s^{1/2}Bs^{1/2}=B$.  
On the other hand, the definition $B=s^{1/2}As^{1/2}$ gives
$A=B$ as subalgebras of the canonically identified biduals.
It follows that $s\in M(B)=M(A)$, contradicting the choice of $s$.
Therefore,
\[
    \widetilde T^{**}(1_{A^{**}})=s\notin M(B).
\]
\end{example}

Now we conclude the following, which is the generalization of Theorem \ref{unital}.
\begin{theorem}\label{thecollary}
Let $A$ and $B$ be $C^*$-algebras and let $ T: A_+\to B_+$ be a positively homogeneous order isomorphism with $\widetilde T^{**}(1_{A^{**}}) \in M(B)$. 
Then there is a Jordan $*$-isomorphism $J$ from $A$ onto $B$ such that 
\[
 T(a)=\widetilde T^{**}(1_{A^{**}})^{\frac{1}{2}}J(a)\widetilde T^{**}(1_{A^{**}})^{\frac{1}{2}}, \quad a \in A_{+}.
\]
\end{theorem}
\begin{proof}
    Applying Theorem \ref{thm:extension}, Proposition \ref{BrownO}, and Lemma \ref{equi}, we obtain the theorem.
\end{proof}

Suppose that $QM(B)=M(B)$. Let $T: A_{+} \to B_{+}$ be a positively homogeneous order isomorphism. Let $s:=\widetilde{T}^{**}(1_{A^{**}})$.  By Proposition~\ref{BrownO}, it follows that $s\in M(B)$.
Lemma~\ref{equi} then yields $J(A)=B$. Various classes of $C^*$-algebras $B$ satisfying $QM(B)=M(B)$ are known (see~\cite{BrownQM} for
further details). Among these, we now describe several cases in which the precise form
of the positively homogeneous
order isomorphisms can be determined. 
\begin{example}[The commutative case]\label{comm}
    Let $C_0(X)$ and $C_0(Y)$ be commutative $C^{*}$-algebras.  Every positively homogeneous order isomorphism $T$ from $C_0^+(X)$ onto $C_0^+(Y)$
has the form
\[
        T(f)(y)=\alpha(y)f(\tau(y)),
\]
where $\tau:Y\to X$ is a homeomorphism and $\alpha:Y\longrightarrow[\delta,\infty)$
is bounded and continuous for some $\delta>0$.
\end{example}
Shibata, Matsuzaki and Miura give another proof of Example \ref{comm} in  \cite[Theorem~1.1]{ShibataMatsuzakiMiura}. 

Since the algebras of all compact operators are self-adjoint standard operator algebras, it follows from \cite[Corollary~2]{MolnarJordan}.
\begin{example}[The algebra of compact operators]
   Let $K(H_1)$ and $K(H_2)$ be the algebras of all compact operators on the Hilbert spaces $H_1$ and $H_2$, respectively. Every positively homogeneous order isomorphism $T$ from $K(H_1)_{+}$ onto $K(H_2)_{+}$
has the form
\[
    T(a)=s^{\frac{1}{2}}UaU^*s^{\frac{1}{2}},  \  a\in K(H_1)
    \quad\text{or}\quad
    T(a)=s^{\frac{1}{2}}Ua^{\mathrm t}U^*s^{\frac{1}{2}},
    \ a\in K(H_1),
\]
where $U: H_1\to H_2$ is a unitary operator and $s \in B(H_2)$ is a positive invertible operator. Note that  $a^t$ denotes the transpose of $a \in K(H_1)$ with respect to a fixed orthonormal basis of $H_1$.
\end{example}

\section{Preservers of the norm of the arithmetic mean}
The following lemma is an  extension of an order characterization originating in Dong, Li, Molnár, and Wong \cite[Lemma~2.6]{DongLiMolnarWong}. Their result treats the case $r=0$ for $C^*$-algebras, and the basic idea of the proof below follows the argument used there. 
\begin{lemma}\label{norm0}
    Let $C$ be a $C^{*}$-algebra. For $a,b\in C_+$ and $r\geq0$, we have
   \[ a-b\leq r1_{C^{**}} \quad\Longleftrightarrow\quad \|a+c\|\leq\|b+c\|+r \quad \text{for all }c\in C_+. \] 
\end{lemma}

\begin{proof}
    ($\Leftarrow$) Let $a,b\in C_+$ be such that $a+b\neq0$. Let $t=\|a+b\|$ and $h=\frac{a+b}{t}$. The support projection of $h$, denoted by $p$, is the smallest projection $p\in C^{**}$ such that $ph=h=hp$. As $0\leq h\leq1$, we have 
    \[ 
    th^{1/n}\geq th=a+b\geq b.
    \]
    This implies 
\[ c_n=th^{1/n}-b\in C_+,\qquad \|b+c_n\|=t. \]
Thus we have $ 0\leq a+th^{1/n}-b\leq(t+r)1_{C^{**}}$ and $h^{1/n}\to p$ in the strong operator topology. Hence we obtain 
\[ a+tp-b\leq(t+r)1_{C^{**}}. \]
Multiplying this inequality on both sides by $p$, we get
\[ a-b\leq rp\leq r1_{C^{**}} \]
since $pa=ap=a$ and $pb=bp=b$. The case $a+b=0$ is immediate. 
 ($\Rightarrow$) If $a-b\leq r1_{C^{**}}$ holds, then $a+c \leq b+c+r1_{C^{**}}$ for any $c \in C_{+}$. Thus we have 
 \[
 \|a+c\| \leq \|b+c+r1_{C^{**}}\|=\|b+c\|+r.
 \]

 \end{proof}
\begin{proposition}\label{norm1}
 Let $C$ be a $C^{*}$-algebra. For $a,b\in C_+$, we have
 \[
  \|a-b\| = \sup_{c\in C_+} \bigl|\|a+c\|-\|b+c\|\bigr|
 \]
\end{proposition}
\begin{proof}
    Since  we have
    \[
    \bigl|\|a+c\|-\|b+c\|\bigr| \le \|a-b\|
    \]
   for any $c \in C_{+}$, we get 
    \[
 \sup_{c\in C_+} \bigl|\|a+c\|-\|b+c\|\bigr| \le \|a-b\|.
 \]
 Let $r:=\sup_{c\in C_+} \bigl|\|a+c\|-\|b+c\|\bigr|$. Then we have
 \[
 \|a+c\|-\|b+c\| \le r \quad \text{and} \quad \|b+c\|-\|a+c\| \le r
 \]
 for any $c \in C_{+}$. By Lemma \ref{norm0}, we get
 \[
 a-b \le r1_{C^{**}} \quad \text{and} \quad b-a \le r1_{C^{**}}.
 \]
 Thus $\|a-b\| \le r$.
\end{proof}

Throughout this paper, $A$ and $B$ are not assumed to be unital. Thus the following theorem is a generalization of \cite[Theorem 2.5]{DongLiMolnarWong}. 
\begin{theorem}\label{additive}
    Let $A$ and $B$ be $C^{*}$-algebras. Assume that $T:A_{+} \to B_{+}$ is a surjective map. Then $T$ satisfies
    \begin{equation}\label{ad1}
        \|T(a)+T(b)\|=\|a+b\|, \quad a, b \in A_{+}
    \end{equation}
    if and only if there is a Jordan $*$-isomorphism $J:A \to B$ which extends $T$.
\end{theorem}

\begin{proof}
Suppose there is a Jordan $*$-isomorphism $J:A \to B$ such that $T=J$ on $A_{+}$. Then we get
\[
 \|T(a)+T(b)\|=\|J(a)+J(b)\|=\|J(a+b)\|=\|a+b\|, \quad a, b \in A_{+}.
\]
We now prove the converse. Applying Proposition \ref{norm1}, for any $a, b \in A_{+}$, we have
\[ \begin{aligned} \|T(a)-T(b)\| &=\sup_{z\in B_+} \bigl|\|T(a)+z\|-\|T(b)+z\|\bigr|\\ &=\sup_{c\in A_+} \bigl|\|a+c\|-\|b+c\|\bigr|=\|a-b\|. \end{aligned}  \]
Thus $T:A_{+} \to B_{+}$ is an isometry. Taking $a=b=0$ in \eqref{ad1} gives $T(0)=0$. Hence $T$ is a bijection.  

Let $a, b \in A_{+}$ with $a \le b$. By Lemma \ref{norm0}, we have $\|a+c\| \le \|b+c\|$ for any $c \in A_{+}$. Thus $\|T(a)+T(c)\|\le \|T(b)+T(c)\|$ for any $c \in A_{+}$. Since $T$ is bijective and Lemma \ref{norm0}, we have $T(a) \le T(b)$. Conversely, the same argument shows that $T(a)\le T(b)$ implies
$a\le b $. Therefore, $T$ is also an order isomorphism. Fix $a, b \in A_{+}$ with $a<b$. We write $[a, b]:=\{ x \in A_{+} \mid a \le x \le b\}$. Define $R_1:[a, b] \to [a, b]$ by $R_1(x)=a+b-x$.  Then $R_1$ is an onto isometry. For any onto isometry $f$ on $[a, b]$, since $f^{-1}R_af$ is also a surjective isometry, we get 
\[
\left\|f^{-1}R_af\left(\frac{a+b}{2}\right)-\frac{a+b}{2}\right\|=2\left\|f\left(\frac{a+b}{2}\right)-\frac{a+b}{2}\right\|.
\]
If $f(\frac{a+b}{2}) \neq \frac{a+b}{2}$, then this contradicts the boundedness of $[a, b]$.  Thus $f\left(\frac{a+b}{2}\right)=\frac{a+b}{2}$ holds for any onto isometry $f$ on $[a, b]$. We define $R_{2}(x)=T(a)+T(b)-x$ for any $x \in [T(a), T(b)]:=\{x \in B_{+} \mid T(a) \le x \le T(b)\}$. Then $R_2$ is an onto isometry from $[T(a), T(b)]$ onto itself. As $T$ is a isometric order isomorphism, $T|_{[a,b]}: [a,b] \to [T(a), T(b)]$ is an onto isometry. Hence $T^{-1}R_2T$ is an onto isometry from $[a,b]$ onto itself.  Thus $T^{-1}R_{2}T\left(\frac{a+b}{2}\right)=\frac{a+b}{2}$. This implies
\[
T\!\left(\frac{a+b}{2}\right) =\frac{T(a)+T(b)}2.
\]
By a routine argument, we see that $T$ is a positively homogeneous
order isomorphism.
By Theorem~\ref{thm:extension} and
Lemma~\ref{lem:bounded-bidual}, $T$ extends to a bounded
complex-linear order isomorphism
$\widetilde T:A\to B$, and
$\widetilde T^{**}:A^{**}\to B^{**}$ is an order isomorphism.
Set
\[
    s:=\widetilde T^{**}(1_{A^{**}}).
\]
Let $(e_\lambda)_\lambda$ be a positive
contractive approximate identity for $A$. Since $e_\lambda \to 1_{A^{**}}$
and  
$\widetilde T^{**}$ is weak$^*$-continuous. We have
\[
T(e_\lambda)=\widetilde{T}^{**}(e_\lambda) \xrightarrow{w^*} \widetilde{T}^{**}(1_{A^{**}})=s.
\]
As $T$ is an order isomorphism and isometry with $T(0)=0$, we have $0 \le T(e_\lambda) \le 1_{B^{**}}$ for each $\lambda$.  
Since the positive cone of $B^{**}$ is weak$^*$ closed, we get
\[
0\leq s \leq 1_{B^{**}}.
\]
Now let $(f_\mu)_\mu$ be a positive contractive approximate identity for
$B$. As $T$ is surjective, for each $\mu$ there exists
$a_\mu\in A_+$ such that $f_\mu=T(a_\mu)$.
Since $T$ is an isometry and $T(0)=0$, we obtain
$\| a_\mu\|=\| f_\mu\| \leq 1$.
Thus, in $A^{**}$, we have $0\leq a_\mu\leq 1_{A^{**}}$. 
Since $\widetilde{T}^{**}$ is an order isomorphism, we have
\[
0\leq f_\mu=T(a_\mu)=\widetilde{T}^{**}(a_\mu)
\leq \widetilde{T}^{**}(1_{A^{**}})=s.
\]
Since $f_\mu\xrightarrow{w^*}1_{B^{**}}$ and the positive cone of
$B^{**}$ is weak$^*$ closed, we get 
\[
1_{B^{**}}\leq s.
\]
Therefore, we conclude that $\widetilde{T}^{**}(1_{A^{**}})=1_{B^{**}} \in M(B)$. 
By Theorem \ref{thecollary}, there is a Jordan $*$-isomorphism $J:A \to B$ such that  $T(a)=J(a)$ for any $a \in A_{+}$. 
\end{proof}

\subsection*{Declaration on the Use of Generative AI}

ChatGPT (OpenAI) was used during the preparation of this manuscript for language
editing, organization of the exposition, and discussion of mathematical
arguments.  The authors independently checked all mathematical statements,
proofs, and references, edited the resulting text, and take full responsibility
for the content of the manuscript.

\end{document}